\documentclass[12pt]{amsart}
\usepackage{amscd,amsmath,amsthm,amssymb}
\usepackage[left]{lineno}
\usepackage{pst-plot,pst-3d}
\usepackage[T1]{fontenc}
\usepackage{color}
\usepackage{pstricks}
\usepackage{stmaryrd}
\usepackage[utf8]{inputenc}
\usepackage{pstricks-add}
\usepackage{graphicx}
\usepackage{epstopdf}
\usepackage{graphicx}
\usepackage{cleveref}
\usepackage{leftidx}
\usepackage{tikz}
\usepackage{tikz-cd}
\usepackage{url}

\definecolor{verylight}{gray}{0.97}
\definecolor{light}{gray}{0.9}
\definecolor{medium}{gray}{0.85}
\definecolor{dark}{gray}{0.6}

 \def\frk{\mathfrak}               

 \def\mm{{\frk m}}

 \def\G{{\mathcal G}}

 \def\Tet{\textup{teter}}

 \def\xb{{\mathbf x}}

 \def\0b{{\mathbf 0}}
\def\gb{{\mathbf g}}

 \def\opn#1#2{\def#1{\operatorname{#2}}} 
 \opn\chara{char} \opn\length{\ell} \opn\pd{pd} \opn\rk{rk}
 \opn\projdim{proj\,dim} \opn\injdim{inj\,dim} \opn\rank{rank}
 \opn\depth{depth} \opn\grade{grade} \opn\height{height}
 \opn\embdim{emb\,dim} \opn\codim{codim}
 
 \opn\Tr{Tr} \opn\bigrank{big\,rank}
 \opn\superheight{superheight}\opn\lcm{lcm}
 \opn\trdeg{tr\,deg}
 \opn\reg{reg} \opn\lreg{lreg} \opn\ini{in} \opn\lpd{lpd}
 \opn\size{size} \opn\sdepth{sdepth}
 \opn\link{link}\opn\fdepth{fdepth}\opn\lex{lex}
 \opn\tr{tr}
 \opn\type{type}
 \opn\gap{gap}
 \opn\arithdeg{arith-deg}
 \opn\HS{HS}
 \opn\tet{tet}
 \opn\div{div} \opn\Div{Div} \opn\cl{cl} \opn\Cl{Cl}
 \opn\Spec{Spec} \opn\Supp{Supp} \opn\supp{supp} \opn\Sing{Sing}
 \opn\Ass{Ass} \opn\Min{Min}\opn\Mon{Mon}
 \opn\Ann{Ann} \opn\Rad{Rad} \opn\Soc{Soc}\opn\Deg{Deg} \opn\Gen{Gen}\opn{\Im}{Im}
 \opn\Im{Im} \opn\Ker{Ker} \opn\Coker{Coker} \opn\Am{Am}
 \opn\Hom{Hom} \opn\Tor{Tor} \opn\Ext{Ext} \opn\End{End}
 \opn\Aut{Aut} \opn\id{id}
 
 \opn\nat{nat}
 \opn\pff{pf}
 \opn\Pf{Pf} \opn\GL{GL} \opn\SL{SL} \opn\mod{mod} \opn\ord{ord}
 \opn\Gin{Gin} \opn\Hilb{Hilb}\opn\sort{sort}
 \opn\PF{PF}\opn\Ap{Ap}
 \opn\mult{mult}
 \opn\bight{bight}
  \opn\bg{bg}
   \opn\gcl{gcl}
 \opn\aff{aff}
 \opn\relint{relint} \opn\st{st}
 \opn\lk{lk} \opn\cn{cn} \opn\core{core} \opn\vol{vol}  \opn\inp{inp} \opn\nilpot{nilpot}
 \opn\link{link} \opn\star{star}\opn\lex{lex}\opn\set{set}
 \opn\width{wd}
 \opn\Fr{F}
 \opn\QF{QF}
 \opn\G{G}
 \opn\type{type}\opn\res{res}
 \opn\conv{conv}
 \opn\Ind{Ind}
 \opn\soc{soc}
 \opn\gr{gr}
 
 \def\pot#1#2{#1[\kern-0.28ex[#2]\kern-0.28ex]}

 \opn\dirlim{\underrightarrow{\lim}}
 \opn\inivlim{\underleftarrow{\lim}}
 \let\sect=\cap
 
 \let\tensor=\otimes
 \let\iso=\cong

 \let\Dirsum=\bigoplus
 
 \let\to=\rightarrow
 \let\To=\longrightarrow
 \def\Implies{\ifmmode\Longrightarrow \else
         \unskip${}\Longrightarrow{}$\ignorespaces\fi}
 \def\implies{\ifmmode\Rightarrow \else
         \unskip${}\Rightarrow{}$\ignorespaces\fi}
 \def\iff{\ifmmode\Longleftrightarrow \else
         \unskip${}\Longleftrightarrow{}$\ignorespaces\fi}

 \let\:=\colon
 \newtheorem{Theorem}{Theorem}[section]
 \newtheorem{Lemma}[Theorem]{Lemma}
 \newtheorem{Corollary}[Theorem]{Corollary}
 \newtheorem{Proposition}[Theorem]{Proposition}
 \theoremstyle{definition}
 \newtheorem{Remark}[Theorem]{Remark}
 
 \newtheorem{Example}[Theorem]{Example}

 \let\epsilon\varepsilon
 \let\kappa=\varkappa
 \opn\dis{dis}
 \def\pnt{{\raise0.5mm\hbox{\large\bf.}}}
 
 \opn\Lex{Lex}

\begin{document}

\title{The canonical trace of determinantal rings}
\author{Antonino Ficarra,  J\"urgen Herzog,  Dumitru I. Stamate and Vijaylaxmi Trivedi}
\date{\today }

\address{Antonino Ficarra, Department of mathematics and computer sciences, physics and earth sciences, University of Messina, Viale Ferdinando Stagno d'Alcontres 31, 98166 Messina, Italy}
\email{antficarra@unime.it}

\address{J\"urgen Herzog, Fakult\"at f\"ur Mathematik, Universit\"at Duisburg-Essen, 45117
Essen, Germany} \email{juergen.herzog@uni-essen.de}

\address{Dumitru I. Stamate, Faculty of Mathematics and Computer Science, University of Bucharest, Str. Academiei 14, Bucharest -- 010014, Romania }
\email{dumitru.stamate@fmi.unibuc.ro}

\address{Vijaylaxmi Trivedi, School of Mathematics, Tata Institute of Fundamental Research, Homi Bhabha Road,
	Mumbai-400005, India}
\email{vija@math.tifr.res.in}

\thanks{This paper was written while the first, the third and the fourth author visited the Faculty of Mathematics of Essen. D.I. Stamate was partly supported by
	a grant of the Ministry of Research, Innovation and Digitization, CNCS - UEFISCDI, project number PN-III-P1-1.1-TE-2021-1633, within PNCDI III.}


\maketitle

\begin{abstract}
We compute the canonical trace of generic determinantal rings  and provide a sufficient condition for the trace to specialize. As an application we determine  the canonical trace $\tr(\omega_R)$ of a Cohen--Macaulay ring  $R$ of codimension two, which is generically Gorenstein.  It is shown that if the defining ideal $I$ of $R$ is generated by $n$ elements, then $\tr(\omega_R)$ is generated  by the $(n-2)$-minors of the  Hilbert-Burch matrix of $I$.
\end{abstract}

\section*{Introduction}

Let $(R,\mm)$ be a Cohen--Macaulay local ring which admits a  canonical module or 
a finitely generated graded Cohen--Macaulay $K$-algebra with graded maximal ideal $\mm$.  In both cases we denote by $\omega_R$ the canonical module of $R$.

The {\em trace} of an $R$-module $M$ is defined to be the ideal
\[
\tr_R(M)=\sum_{\varphi\in\Hom_R(M,R)}\varphi(M)\subseteq R.
\]
We omit the index $R$ in $\tr_R(M)$, if it is clear from the context in which ring $R$ we are computing the trace. The trace of $\omega_R$ is called the  {\em canonical trace} of $R$. There is a particular interest in  the canonical trace because it  determines the non-Gorenstein locus of $R$. It also allows to define nearly Gorenstein rings, which are the rings for which  $\mm\subseteq \tr_R(\omega_R)$.   This class of rings have first been considered in \cite{HV}. The name “nearly Gorenstein” was introduced in \cite{HHS1}. In 1974 William  Teter \cite{T} studied  $0$-dimensional local rings which can be represented as Gorenstein rings modulo their socle. Such rings are nowadays called Teter rings.  It has been shown (see \cite{T},\cite{HV} and \cite{ET})  that $R$ is  a Teter ring if and only if there exists an epimorphism $\varphi\: \omega_R\to \mm$. This result shows  that a Teter ring is nearly Gorenstein. Inspired by this result, the authors of \cite{GHHM} define the Teter number $\Tet(R)$ of $R$ as the smallest integer $t$ for which there exist $R$-module homomorphisms $\varphi_1, \ldots, \varphi_t\: \omega_R\to \tr(\omega_R)$ with $\tr(\omega_R)=\sum_{i=1}^t\varphi_i(\omega_R)$.  For further studies of the canonical trace we refer the reader to \cite{D}.

In this paper we study the canonical trace  of determinantal rings. Let $K$ be a field,  and let $X$ be an $m\times n$ matrix  of indeterminates with $m<n$. Furthermore,  assume that  $1<r+1\le m$ and let $R=K[X]/I_{r+1}(X)$, where $I_{r+1}(X)$ is the ideal of $K[X]$ generated by the $(r+1)$-minors of $X$. Section~1 of this paper is devoted to prove that $\tr(\omega_R)=I_{r}(X)^{n-m}R$, see Theorem~\ref{generic}. As a consequence,  the non-Gorenstein locus and the singular locus of $R$ coincide.  

In Section~2 we determine the Teter number of $R$. It can be easily seen that in general the Teter number of a ring is bounded above by the minimal number of  generators of the anti-canonical module $\omega_R^{-1}$. We show in Theorem~\ref{Thm:TetGeneric} that  for our determinantal rings this upper bound is reached,  and  hence  by  \cite[Proposition 4.1]{BRW} it is equal to the determinant of $\big[\binom{2n-m-j}{n-i}\big]_{1\le i,j\le r}$.

It is natural to ask whether the canonical trace specializes. It is known and easily seen that if $\xb $ is an  $R$-sequence and  $\overline{R}=R/\xb R$, then $\tr(\omega_R)\overline{R}\subseteq \tr(\omega_{\overline{R}})$. If equality holds, we say that  canonical trace specializes. In Section~3 we give an explicit example which shows that the trace does not always specialize.  On the other hand, we prove in Theorem~\ref{reduction} that the canonical trace specializes if $R$ and $\overline{R}$ are generically Gorenstein and $\omega_R^2$ is a Cohen--Macaulay $R$-module. For our determinantal rings this condition on $\omega_R^2$ is satisfied if and only if $n\leq 2m-r$, see \cite[Theorem 4.3]{BRW}. We should stress the fact that Theorem~\ref{reduction} only gives a sufficient condition for the canonical trace to specialize. Indeed, we don't have any example of a determinantal ring whose canonical trace does not specialize. 

Let $I$ be graded perfect ideal  of height $2$ in the polynomial ring $S=K[x_1,\ldots, x_s]$. Assume that $I$ is minimally generated by $n$ homogeneous  elements and that $R=S/I$ is generically Gorenstein. Theorem~\ref{reduction} is used in Corollary ~\ref{codim2} to show that 
$
\tr(\omega_R)=I_{n-2}(A)R
$, 
where $A$ is a Hilbert-Burch matrix of $I$.  
This result is applied to rings defined by  generic perfect monomial ideals of grade $2$ and also to numerical semigroup rings generated by $3$ 
elements,  to recover a result in \cite{HHS2}. 

\section{The canonical trace of determinantal rings}

The aim of this section is to prove

\begin{Theorem}
	\label{generic}
	Let $K$ be a field, and let $X=(x_{ij})$ be an $m\times n$ matrix of indeterminates with $m\leq n$. Let $1\leq r+1\le m$ and $R=K[X]/I_{r+1}(X)$. Then 
	\[
	\tr(\omega_R)=I_{r}(X)^{n-m}R.
	\]
\end{Theorem} 

The proof of this theorem needs some preparations. Let $X=(x_{ij})$ be an $m\times n$ matrix of indeterminates with $m\leq n$. Let $A\subseteq[m],B\subseteq[n]$ with $|A|=|B|=s$, and assume that  $A=\{a_1<\ldots<a_{s}\}$ and $B=\{b_1<\ldots<b_{s}\}$.  We denote by $[A|B]=[a_1,\dots,a_s|b_1,\dots,b_s]$ the $r$-minor of $X$, whose $i$-th row is the $a_i$-th row of $X$ and whose $j$-th column is the $b_j$-th  column of $X$.  

Let  $[C|D]=[c_1,\dots,c_t|d_1,\dots,d_t]$ be any other minors of $X$.  We set
$$
[A|B]\le[C|D]\ \iff\ s\ge t\ \text{and}\ a_i\le c_i,\ b_i\le d_i,\ i=1,\dots,t.
$$
Then $\le$ is a partial order on the set of all minors of $X$.

Let $r+1\leq m$. It is well-known that $K[x_{ij}]/I_{r+1}(X)$ is an \textit{algebra with straightening laws}. In particular, if $u,v$ are incomparable $r$-minors, we have a straightening relation
\begin{equation}\label{eq:incomLasagna}
u\cdot v=\sum_i \lambda_i u_iv_i,
\end{equation}
where $\lambda_i\in K$, $\lambda_i\ne0$, and for all $i$, $u_i,v_i$ are $r$-minors with $u_i\le v_i$, $u_i<u$ and $u_i<v$.

We denote again by $X$ the image of the matrix $X$ in $R$. Let $P$ be the ideal of $R$ generated by the $r$-minors of the first $r$ rows of $X$. Likewise, let $Q$ be the ideal of $R$ generated by the $r$-minors of the first $r$ columns of $X$. Furthermore, we denote by $\delta$ the minor $[1,\dots,r|1,\dots,r]$.

If $I$ is a graded ideal in $R$, $\mu(I)$ denotes the minimal number of generators. 
\begin{Lemma}\label{Lemma:Lasagna}
	With the notation introduced, we have for all $\ell\ge1$,
	$$
	\mu((PQ)^\ell)=\mu(P^\ell)\mu(Q^\ell).
	$$
\end{Lemma}
\begin{proof}
	Let $Y=(y_{ij})$ be an $m\times r$ matrix and $Z=(z_{ij})$ be an $r\times n$ matrix of indeterminates. For an integer $k$, we set $[k]=\{1,\dots,k\}$. Furthermore, if $I\subseteq[m]$ with $|I|=r$, we denote by $Y_I$ the $r\times r$ submatrix  of $Y$ whose rows are indexed by $I$. Similarly, if $J\subseteq[n]$ with $|J|=r$, we denote by $Z_J$ the $r\times r$ submatrix  of $Z$ whose columns are indexed by $J$. We set $\delta_Y=\det(Y_{[r]})$ and $\delta_Z=\det(Z_{[r]})$.
	
	For the proof we use the isomorphism
	\begin{align*}
	\varphi:R\rightarrow K[Y\cdot Z],\ \ x_{ij}\mapsto(Y\cdot Z)_{ij}\ \ \textup{for all}\ \ i\in[m],\ j\in[n].
	\end{align*}
	Let $u=[1,\dots,r|J]\in P$ be an $r$-minor, with $J\subseteq[n]$, $|J|=r$. Then
	$$
	\varphi(u)=\det((Y\cdot Z)_{i\in[r],j\in J})=\det(Y_{[r]}\cdot Z_{J})=\delta_Y\det(Z_{J}).
	$$
	Thus $\varphi(P)=\delta_Y I_{r}(Z)$. Similarly, we have $\varphi(Q)=\delta_Z I_r(Y)$. This implies that
	$$
	\varphi(P^\ell Q^\ell)=(\delta_Y\delta_Z)^\ell I_r(Y)^\ell I_r(Z)^\ell.
	$$
	Therefore,
	$$
	\mu((PQ)^\ell)=\mu(I_r(Y)^\ell I_r(Z)^\ell)=\mu(I_r(Y)^\ell)\mu(I_r(Z)^\ell)=\mu(P^\ell)\mu(Q^\ell).
	$$
	The second equality follows from the fact that the generators of $I_r(Y)$ and $I_r(Z)$ are polynomials in different sets of variables.
\end{proof}
\begin{Corollary}\label{Cor:Lasagna}
	With the notation introduced, we have
	$$
	PQ=\delta I_{r}(X).
	$$
\end{Corollary}
\begin{proof}
	Let $u\in P$ and $v\in Q$. If $u=\delta$ or $v=\delta$, then $uv\in\delta I_{r}(X)$. Suppose both of them are different from $\delta$. Then $u$ and $v$ are incomparable $r$-minors. The only $r$-minor which is less than $u$ and $v$ is $\delta$. Therefore, equation (\ref{eq:incomLasagna}) implies that for each $i$, $u_i=\delta$. This shows that $PQ\subseteq\delta I_r(X)$.
	
	It is clear that $\mu(\delta I_r(X))=\mu(I_{r}(X))=\binom{m}{r}\binom{n}{r}$. Hence, if we show that $\mu(PQ)=\binom{m}{r}\binom{n}{r}$, then equality follows. This follows from the previous lemma with $\ell=1$, because $\mu(P)=\binom{n}{r}$ and $\mu(Q)=\binom{m}{r}$.
\end{proof}

Theorem \ref{generic} is now a consequence of the following more general result, observing that $Q^{n-m}$ is the canonical ideal $\omega_R$ of $R$ \cite[Theorem 7.3.6]{BH}.

\begin{Theorem}
	We have $\tr(Q^\ell)=I_{r}(X)^\ell R$, for all $\ell\ge1$.
\end{Theorem}
\begin{proof}
	By Corollary \ref{Cor:Lasagna} we have $Q^{\ell}(\delta^{-1}P)^{\ell}=I_r(X)^{\ell}$. On \cite[Page 315]{BH} it is noted that $\height(I_r(X)R)>1$. Therefore, for any prime ideal $L$ of $R$ of height one, we have $Q^{\ell}R_L(\delta^{-1}P)^{\ell}R_L=R_L$. This implies that $(Q^{\ell})^{-1}R_L=(\delta^{-1}P)^{\ell}R_L$. Since $Q^{\ell}$ and $(\delta^{-1}P)^{\ell}$ are divisorial ideals and the previous equation holds for all height one prime ideals $L$ of $R$, it follows that $(Q^{\ell})^{-1}=(\delta^{-1}P)^{\ell}$. Here we use the fact that if $I$ is a divisorial ideal, then $I=\bigcap_{\mathfrak{p}}IR_\mathfrak{p}$, where the intersection is taken over all height one prime ideals $\mathfrak{p}$ of $R$.
	
	Since $Q^\ell$ has positive grade, by \cite[Lemma 1.1]{HHS1} $\tr(Q^\ell)=Q^\ell(Q^\ell)^{-1}$, and hence the desired result follows.
\end{proof}

\begin{Corollary}
	\label{equal}
	Suppose the determinantal ring $R$ is not Gorenstein. Then the non-Gorenstein  locus and the singular locus of $R$ coincide.
\end{Corollary}

\begin{proof}
	It is noted in \cite[Lemma 2.1]{HHS1} that the canonical trace determines the non-Gorenstein locus. Our determinantal ring $R$ is Gorenstein if and only if $m=n$. Hence, if $R$ is not Gorenstein, Theorem~\ref{generic}  implies that $\sqrt{\tr(\omega_R)}=I_r(X)R$. By \cite[Proposition 7.3.4]{BH} $I_r(X)R$ determines the singular locus. This implies the assertion. 
\end{proof}

\section{The Teter number of determinantal rings}

In \cite{GHHM} the notion of \textit{Teter ring} was introduced. Let $R$ be a  Cohen--Macaulay ring which is either a local ring with canonical module or a  graded $K$-algebra over a field $K$. Since by definition $\tr(\omega_R)=\sum_{\varphi\in\textup{Hom}_R(\omega_R,R)}\varphi(\omega_R)$, it is natural to determine the smallest number $t$ of maps $\varphi_1,\dots,\varphi_t\in\textup{Hom}_R(\omega_R,R)$ such that $\tr(\omega_R)=\sum_{i=1}^{t}\varphi_i(\omega_R)$. We call such a number, the \textit{Teter number} of $R$, and denote it by $\Tet(R)$. If $\Tet(R)=1$, $R$ is called a ring of \textit{Teter type}.

The purpose of this section is to determine the Teter number of a determinantal ring. For this aim, we adopt the following strategy. Recall that a Cohen--Macaulay $R$ is \emph{generically Gorenstein} if $R_\mathfrak{p}$ is Gorenstein for all minimal prime ideals $\mathfrak{p}$ or $R$. If this is the case, then $\omega_R$ can be identified with an ideal of height one in $R$. Hence $\tr(\omega_R)=\omega_R\cdot\omega_R^{-1}$, by \cite[Lemma 1.1]{HHS1}. Since $\textup{Hom}_R(\omega_R,R)\cong\omega_R^{-1}$, each map $\varphi:\omega_R\rightarrow R$ is multiplication by some element $x\in\omega_R^{-1}$. Hence, the Teter number is the smallest number $t$ such that there exist $x_1,\dots,x_t\in\omega_R^{-1}$ with
$$
x_1\omega_R+\dots+x_t\omega_R=\tr(\omega_R).
$$
Set $J=(x_1,\dots,x_t)$. Therefore, $J\omega_R=\tr(\omega_R)$. Let $\mu(\omega_R^{-1})$ be the minimal number of generators of $\omega_R^{-1}$. Since $\omega_R^{-1}\omega_R=\tr(\omega_R)$, our discussion yields the following upper bound:
\begin{Proposition}
	Let $R$ be a ring as introduced above. Then $\Tet(R)\le\mu(\omega_R^{-1})$.
\end{Proposition}

Below we discuss a special case where this upper bound is reached.
\begin{Lemma}\label{Lemma:TeterMu}
	Let $R$ be a ring as introduced above. Assume that $\mu(\tr(\omega_R))=\mu(\omega_R)\mu(\omega_R^{-1})$. Then $\Tet(R)=\mu(\omega_R^{-1})$.
\end{Lemma}
\begin{proof}
	Suppose by contradiction that $\Tet(R)<\mu(\omega_R^{-1})$. Then, we can find $J\subseteq\omega_R^{-1}$ with $\mu(J)<\mu(\omega_R^{-1})$ and $J\omega_R=\tr(\omega_R)$. Therefore,
	$$
	\mu(\tr(\omega_R))\le\mu(J)\mu(\omega_R)<\mu(\omega_R^{-1})\mu(\omega_R)=\mu(\tr(\omega_R)).
	$$
	A contradiction. Since we also have $\Tet(R)\le\mu(\omega_R^{-1})$, the conclusion follows. 
\end{proof}

\begin{Theorem}
	\label{Thm:TetGeneric}
	Let $K$ be a field, and let $X=(x_{ij})$ be an $m\times n$ matrix of indeterminates with $m<n$. Let $1<r+1\le m$ and $R=K[X]/I_{r+1}(X)$. Then 
	\[
	\Tet(R)=\det\bigg[\binom{2n-m-j}{n-i}\bigg]_{1\le i,j\le r}.
	\]
\end{Theorem}
\begin{proof}
	By using the notation introduced before Corollary \ref{Cor:Lasagna}, we know that $\omega_R=Q^{n-m}$, $\omega_R^{-1}=(\delta^{-1}P)^{n-m}$ and $\tr(\omega_R)=I_r(X)^{n-m}$. Obviously, $\mu(\omega_R^{-1})=\mu(P^{n-m})$ and $\mu(\tr(\omega_R))=\mu(\delta^{n-m}I_r(X)^{n-m})$. By Lemma \ref{Lemma:Lasagna} $\mu(P^{n-m})\mu(Q^{n-m})=\mu((PQ)^{n-m})$. By Lemma \ref{Lemma:TeterMu} it follows that $\Tet(R)=\mu(P^{n-m})$ and using \cite[Proposition 4.1]{BRW} we get the formula in the statement.
\end{proof}

We conclude this section with the following remark.
\begin{Remark}
	\rm Suppose $R$ is generically Gorenstein and that $\omega_R$ is a divisorial ideal. Let $J\subseteq\omega_R^{-1}$ such that $J\omega_R=\tr(\omega_R)$. Then $J\omega_R=\omega_R^{-1}\omega_R$. Hence,
	$$
	\omega_R:(\omega_RJ)=\omega_R:(\omega_R\omega_R^{-1}).
	$$
	On the other hand $\omega_R:(\omega_RJ)=(\omega_R:\omega_R):J=R:J=J^{-1}$ and $\omega_R:(\omega_R\omega_R^{-1})=(\omega_R:\omega_R):\omega_R^{-1}=R:\omega_R^{-1}=(\omega_R^{-1})^{-1}=\omega_R$, because we assumed that $\omega_R$ is divisorial. Therefore, $J^{-1}=\omega_R$. Because $J$ may not be divisorial, we only obtain the following inclusion,
	$$
	J\subseteq(J^{-1})^{-1}=\omega_R^{-1}.
	$$
	Hence, if $J\omega_R=\tr(\omega_R)$, the divisorial closure $(J^{-1})^{-1}$ of $J$ has to be equal to the anti-canonical module $\omega_R^{-1}$, but $J$ itself may be a proper sub-ideal of $\omega_R^{-1}$. 
\end{Remark}

\section{Reduction of  traces}

Let $R$ be a local ring or a  positively graded  graded $K$-algebra  which is Cohen--Macaulay and admits a canonical module $\omega_R$. Let $\xb=x_1,\ldots,x_m$  be a (graded) sequence of elements in $R$, and set $\overline{M}=M/(\xb)M$. It is easy to see that $\tr(\omega_R)\overline{R}\subseteq \tr(\omega_{\overline{R}})$, if ${\bf x}$ is an $R$-sequence. The question is whether this inclusion may be strict.

When $R$ is generically Gorenstein, then $\omega_R$ may be identified with an ideal, and we may consider the powers  of $\omega_R$.

\begin{Theorem}
	\label{reduction}
	Let $R$ be a ring as introduced above, and assume that $R$ is generically Gorenstein.  Then 
	\[
	\tr(\omega_R)\overline{R}= \tr(\omega_{\overline{R}}),
	\]
	if $\omega_R^2$ is Cohen--Macaulay and $\overline{R}$  is generically Gorenstein as well. 
\end{Theorem}

The proof of the theorem is based on the following 

\begin{Lemma}
	\label{based}
	Let $R$ be a Cohen--Macaulay ring as before and assume that $R$ is generically Gorenstein. Then 
	\[
	\Hom_R(\omega_R^2,\omega_R)\iso \Hom_R(\omega_R, R).
	\]
\end{Lemma}

\begin{proof}
	Let $U$ be the kernel of the canonical map $\omega_R\tensor_R\omega_R\to \omega_R^2$. Then we obtain the exact sequence
	\[
	0\to \Hom_R(\omega_R^2,\omega_R)\to \Hom_R(\omega_R\tensor_R\omega_R,\omega_R)\to \Hom_R(U,\omega_R)
	\]
	Since $R$ is generically Gorenstein it follows that $\omega_{R_P}\iso R_P$ for all minimal prime ideals $P$ of $R$. This implies that $\omega_R\tensor_R\omega_R\to \omega_R^2$ becomes an isomorphism after localization at a minimal prime ideal of $R$. Hence, $U_P=0$ for all minimal prime ideal $P$ of $R$. Since $\Ass \Hom_R(U,\omega_R)= \Supp U\sect \Ass \omega_R$ and since  $\Ass \omega_R=\Ass R$,  we conclude that $\Ass \Hom_R(U,\omega_R)=\emptyset$. This implies that $\Hom_R(U,\omega_R)=0$. Therefore, 
	\[
	\Hom_R(\omega_R^2,\omega_R)\iso \Hom_R(\omega_R\tensor_R\omega_R,\omega_R). 
	\]
	By using that $\Hom_R(\omega_R, \omega_R)\iso R$, we see that 
	\[
	\Hom_R(\omega_R\tensor_R\omega_R,\omega_R)\iso \Hom_R(\omega_R, \Hom_R(\omega_R,\omega_R))\iso \Hom_R(\omega_R,R).
	\]
	Thus, the desired conclusion follows. 
\end{proof}

\begin{proof}[Proof  of Theorem~\ref{reduction}]
	Consider the commutative diagram
	\begin{center}
		\begin{tikzcd}
			\Hom_R(\omega_R,R)\times \omega_R\arrow[r, "\varphi"]\arrow[d, "\alpha"] & R\arrow[d, "\beta"]\\
			\overline{\Hom_R(\omega_R,R)}\times \overline{\omega_R}\arrow[r, "\psi" ]&  \overline{R} 
		\end{tikzcd}
	\end{center}
	with its natural maps. Then $\Im(\varphi)=\tr(\omega_R)$, 
	and hence $\Im(\beta\circ\varphi)=\tr(\omega_R)\overline{R}$. It follows that $\Im(\psi)=\tr(\omega_R)\overline{R}$. 
	
	Lemma~\ref{based}, implies that 	$\Hom_R(\omega_R,R)\iso \Hom_R(\omega_R^2,\omega_R)$, since we assume that $R$ is generically Gorenstein. By assumption, $\omega_R^2$ is Cohen--Macaulay, and since $\dim \omega_R^2 =\dim R$, it is a maximal Cohen--Macaulay module.
	
	Thus, \cite[Theorem~3.3.10(c)(ii)]{BH} implies that $\Ext^i_R(\omega_R^2,\omega_R)=0$ for $i>0$, and hence  it follows from \cite[Proposition~3.3.3]{BH} that 
	\[
	\overline{\Hom_R(\omega_R^2,\omega_R)}\iso \Hom_{\overline{R}} (\overline{\omega_R^2},\overline{\omega_R})\iso
	\Hom_{\overline{R}}(\omega^2_{\overline{R}},
	\omega_{\overline{R}}).
	\]
	Now we use that $\overline{R}$ is generically Gorenstein, and apply again Lemma~\ref{based} to obtain that $\Hom_{\overline{R}}(\omega^2_{\overline{R}},
	\omega_{\overline{R}})\iso \Hom_{\overline{R}}(\omega_{\overline{R}},\overline{R})$. This shows that $\Im \psi= \tr(\omega_{\overline{R}}), $ and  completes  the proof of the theorem. 
\end{proof}

In the following example, using \textit{Macaulay2} \cite{GDS} we verified that the canonical trace does not specialize in general.
\begin{Example}
	\rm Consider the monomial ideal $I=(x_1y_1,x_2y_2,x_3y_3,x_1x_2,x_2y_3,x_1x_3)\subset S=K[x_1,x_2,x_3,y_1,y_2,y_3]$. Then $R=S/I$ is a Cohen--Macaulay $K$-algebra and the canonical trace of $R$ is
	$$
	\tr(\omega_R)=(x_{1},x_{2},x_{3}^{2},x_{3}y_{1},x_{3}y_{2},y_{1}y_{2},y_{1}y_{3},y_{2}y_{3},y_{3}^{2})R.
	$$
	The element $x_1-y_1$ is regular on $R$. We set $\overline{R}=R/(x_1-y_1)R$. Then,
	$$
	\tr(\omega_{\overline{R}})=(x_{2},x_{3}^{2},x_{3}y_{2},y_{1},y_{3})\overline{R}.
	$$
	Hence, the canonical trace does not specialize in such a case. In this example $R$ and $\overline{R}$ are both generically Gorenstein,  but $\omega_R^2$ is not Cohen--Macaulay, because $\depth \omega_R^2=2$,  while $\depth R=3$. 
\end{Example}

Let $K$ be a field, and  let $S=K[x_1,\ldots,x_s]$ be the graded polynomial ring with $\deg x_i=d_i>0$  for $i=1,\ldots,s$. Furthermore, let $A=(f_{ij})$ be an $m\times n$-matrix of homogeneous polynomials  of $S$ such that $\deg f_{ij}=a_i-b_j$, where   $a_1,\ldots,a_m$ and $b_1,\ldots,b_n$ positive integers  and where $f_{ij}=0$ if  $a_i-b_j\leq 0$. Under these assumptions all minors of $A$ are homogeneous polynomials. Hence $R=S/I_{r+1}(A)$ is a graded $K$-algebra. We have  $\height I_{r+1}(A)\leq 
\height  I_{r+1}(X)=(n-r)(m-r)$. If equality holds, then $R$ is Cohen--Macaulay.

\begin{Corollary}
	\label{when}
	With the notation introduced, assume that $\height  I_{r+1}(A)=(n-r)(m-r)$, that $R$ is generically Gorenstein and that $n\leq 2m-r$. Then
	\[
	\tr(\omega_R)=I_{r}(A)^{n-m}R.
	\] 
\end{Corollary}	

\begin{proof}
	Let $X$ be the $m\times n$ matrix whose  entries are  the indeterminates $x_{ij}$,   $i=1,\ldots,m$ and $j=1,\ldots,n$, and let $S'=S[X]$ and $R'=S'/I'$, where $I'=I_{r+1}(X)S'$. We set $\deg x_{ij}=b_j-a_i$ for all $i$ and $j$. Then $I'$ is a  graded ideal in the non-standard graded polynomial ring $S'$,  and  $R$ is a specialization of $R'$. In other words, the sequence $\gb=x_{11}-f_{11}, \ldots, x_{n-1,n}-f_{n-1,n}$  is a homogeneous  regular sequence on $S'$ and $R'$ with $S'/(\gb)=S$ and $R'/(\gb)=R$. 
	
	The ring $R'$ is generically Gorenstein, since it is a  Cohen--Macaulay domain, and $R$ is generically Gorenstein,  by assumption. Therefore, by Theorem~\ref{reduction}, $\tr(\omega_R)=\tr(\omega_{R'}) R= I_{r}(A)^{n-m}R$, if $\omega_{R'}^2$ is Cohen--Macaulay. This condition for  $\omega_{R'}$  is obviously  satisfied  if and only if the square of the canonical module $Q^{n-m}$ of $K[X]/I_{r+1}(X)$ is Cohen--Macaulay. By \cite[Theorem 4.3]{BRW} this is the case if and only if $2(n-m)\leq n-r$, equivalently if $n\leq 2m-r$. This completes the proof of the corollary.  
\end{proof}

As an application, we consider a  graded perfect ideal $I$ of height  $2$  in the polynomial ring $S=K[x_1,\ldots, x_s]$ minimally  generated by $n$ homogeneous elements. Then $R=S/I$ has a graded resolution 
\[
0\to \Dirsum_{j=1}^{n-1} S(-b_j)\stackrel{A}{\To}\Dirsum_{i=1}^n  S(-a_i)\to  R\to 0
\]
with $A$ an $n\times(n+1)$-matrix. 
The Hilbert-Burch theorem \cite[Theorem 1.4.17]{BH} tells us that $I$ is generated by the maximal minors of $A$, which is the same as the maximal minors of $A^T$ - the transpose of $A$. In particular,  $I$ is obtained as a specialization  of the ideal of maximal minors of the generic $(n-1)\times n$ matrix $X$.

The matrix $A$ is called the Hilbert-Burch matrix of $I$ (with respect to the given resolution). 

\begin{Corollary}
	\label{codim2}
	Let  $I$ be a	graded perfect ideal  of height  $2$  in the polynomial ring $S=K[x_1,\ldots, x_s]$ minimally generated by $n$ homogeneous  elements and  with Hilbert-Burch matrix $A$. Let $R=S/I$,  and assume that $R$ is generically Gorenstein.  Then 
	\[
	\tr(\omega_R)=I_{n-2}(A)R.
	\]
\end{Corollary}

\begin{proof}
	Since $R$ is the specialization of $K[X]/I_{n-1}(X)$ with $X$ an $(n-1)\times n$-matrix of indeterminates and since  $r=n-2$, Corollary~\ref{when} yields the desired conclusion.
	\end{proof}

\medskip
As a very special case of Theorem~\ref{reduction} we recover a result in \cite{HHS2}. Let $H$ be a numerical semigroup generated by 3 elements, and let $K$ be any field. The numerical semigroup ring $K[H]$ is a $1$-dimensional (non-standard) graded domain. This implies that the defining ideal $I_H$ of $K[H]$ is a graded perfect height 2 prime ideal in the polynomial ring in 3 variables. By \cite{He}, the ideal $I_H$ is generated by $3$ elements, so that its Hilbert-Burch matrix $A_H$ is a $2\times 3$ matrix. Thus, Corollary~\ref{codim2} implies that the canonical trace  of $K[H]$ is generated by the entries of $A_H$ modulo $I_H$, see \cite[Proposition 3.1]{HHS2} for more details. 

\medskip
In \cite[Remark 6.3]{BH1} it is observed  that any perfect monomial ideal of height $2$ arises from a generic perfect monomial ideal of height 2  by specialization. The family of generic perfect monomial ideals of height 2 are parameterized by trees. This has been worked out in detail in \cite{Na}, whose presentation we follow here. To each tree $\Gamma$ on the vertex set  $[n]=\{1,\ldots,n\}$ one assigns an $(n-1)\times n$-matrix $A(\Gamma)=(a_{ij})$, whose entries are either variables or $0$.  One chooses a total order of the edges of $\Gamma$, and assigns to the $k$th edge $\{i,j\}$ of $\Gamma$, $i<j$, the $k$th row of $A(\Gamma)$ as follows: 
\begin{eqnarray*}
	a_{k\ell}=\begin{cases}
		-x_{ij}, &\text{if $\ell=i$}\\ 
		x_{ji}, & \text{if $\ell=j$}\\
		0, & \text{otherwise}.
	\end{cases}
\end{eqnarray*}
Let $I(\Gamma) =I_{n-1}(A(\Gamma))$. Then $I(\Gamma)$ is a perfect    squarefree monomial ideal of height $2$, whose  Hilbert-Burch matrix is the transpose of $A(\Gamma)$. 
The generators of $I(\Gamma)$ are determined as follows: let $i, j$ be two distinct
vertices of $\Gamma$. Then there exists a unique path from $i$ to $j$, that is, a sequence of pairwise distinct numbers $i = i_0, i_1, i_2, \ldots , i_{k-1}, i_k = j$, where $\{i_\ell,i_{\ell+1}\}$ is an edge for $\ell=1,\ldots,k-1$. We set $b(i,j)=i_1$.  Let $v_j$  be the minor of $A(\Gamma)$ which is obtained by omitting the $j$th column of $A(\Gamma)$. Then 
\[
v_j =\pm\prod_{i=1\atop i\neq j}^n x_{ib(i,j)}
\]
for $j=1,2,\ldots,n$, and $I(\Gamma)=(v_1,\ldots,v_n)$. 

The ideal $I(\Gamma)$ is a squarefree monomial ideal, and therefore generically Gorenstein. It follows from  Corollary~\ref{codim2} that the   $(n-2)$-minors of $A(\Gamma)$  modulo $I(\Gamma)$ generate the canonical trace of $R=S/I(\Gamma)$, where $S$ is the polynomial ring over $K$ generated by the entries of $A(\Gamma)$. Expanding the $(n-1)$-minor $v_j$ along the $i$th row, we see that the $n-2$ minors of $A(\Gamma)$ are the monomials 
\[
v_j/x_{ib(i,j)}
\]
with $j=1,\ldots,n$, $i= 1,\ldots n$ and $i\ne j$. 

In general, these monomials do not form a minimal set of generators of the canonical trace.

For a monomial ideal $I$ we denote by $I(x)$ the monomial ideal which is obtained from $I$ by substituting $x$ by $1$ in each  monomial generator of $I$. This operation is a special case of monomial localization. The above discussions give us the following result.

\begin{Corollary}
	\label{monloc}
	Let $\Gamma$ be a tree on the vertex set $[n]$, and let $R=S/I(\Gamma)$, where $S$ is the polynomial ring over $K$ generated by the entries of $A(\Gamma)$. Then 
	\[
	\tr(\omega_R)=\big(\sum I(\Gamma)(x_{ij})\big)/I(\Gamma),
	\]
where the sum is taken over all $i$ and $j$ for which $x_{ij}$ is a variable appearing in $A(\Gamma)$.
\end{Corollary}

\begin{Example}
Let $\Gamma$ be the tree with the edges $\{1,2\},\{2,3\}, \{3,4\},\{3,5\}$.  Then  
\begin{eqnarray*}
A(\Gamma)=
\begin{pmatrix}
-x_{12} & x_{21} & 0 & 0 &0\\
0 & -x_{23} & x_{32} & 0 & 0\\
0 & 0 & -x_{34} & x_{43} & 0\\
0 & 0 & -x_{35} & 0 & x_{53}
	\end{pmatrix}
\end{eqnarray*}
and
\begin{eqnarray*}
I(\Gamma)&=& (x_{21}x_{32}x_{43}x_{53}, x_{12}x_{32}x_{43}x_{53},  x_{12}x_{23}x_{43}x_{53},  x_{12}x_{23}x_{34}x_{53}, x_{12}x_{23}x_{43}x_{35}).
\end{eqnarray*}
We write $A(\Gamma)$ modulo $I(\Gamma)$ as 
\begin{eqnarray*}
\begin{pmatrix}
	-a & b & 0 & 0 &0\\
	0 & -c & d & 0 & 0\\
	0 & 0 & -e & f & 0\\
	0 & 0 & -g & 0 & h
\end{pmatrix}.
\end{eqnarray*}
Then the canonical trace of the ring $R$ defined by $I(\Gamma)$ is the ideal of $3$-minors of this matrix, and we obtain
\begin{multline*}
\tr(\omega_R)=(a\,c\,e,\,a\,c\,g,\,a\,c\,f,\,a\,d\,f,\,a\,f\,g,\,b\,d\,f,\,b\,f\,g,\,c\,f\,g,\,a\,c\,h,\,\\
a\,d\,h,\,a\,e\,h,\,b\,d\,h,\,b\,e\,h,\,c\,e\,h,\,a\,f\,h,\,b\,f\,h,\,c\,f\,h,\,d\,f\,h).
\end{multline*}
\end{Example}

{}

\end{document}